\documentclass[reqno, 11pt]{amsart}
\usepackage{pgfplots}
\pgfplotsset{compat=1.17}  
\usepackage{caption}
\usetikzlibrary{intersections} 
\usepgfplotslibrary{fillbetween}

\usepackage{lmodern}
\usepackage[T1]{fontenc}

\makeatletter
\@namedef{subjclassname@2020}{%
  \textup{2020} Mathematics Subject Classification}
\makeatother

\usepackage{ulem}

\usepackage{amsfonts, amsthm, amsmath, amssymb}
\usepackage{hyperref}
\hypersetup{colorlinks=false}

\usepackage{bbm}  

 \usepackage{tabu}

\usepackage[margin=1in]{geometry}

\RequirePackage{mathrsfs} \let\mathcal\mathscr

\usepackage{hyperref}
\usepackage{dsfont}
\usepackage{upgreek}
\usepackage{enumerate}
\usepackage{comment}

\numberwithin{equation}{section}

\newtheorem{theorem}{Theorem}[section] 
\newtheorem{lemma}[theorem]{Lemma}

\theoremstyle{definition}

\newtheorem*{acknowledgements}{Acknowledgements}
\newtheorem{remark}[theorem]{Remark}

\newtheorem{definition}[theorem]{Definition}

\renewcommand{\emph}[1]{\textit{#1}}

\renewcommand{\phi}{\varphi}

\newcommand{\0}{\mathbf{0}}

\renewcommand{\leq}{\leqslant}

\renewcommand{\geq}{\geqslant}

\renewcommand{\c}{\mathbf{c}}

 \renewcommand{\b}{\mathbf{b}}

\renewcommand{\P}{\mathbb{P}}
\newcommand{\Q}{\mathbb{Q}}
\newcommand{\F}{\mathbb{F}}
\newcommand{\N}{\mathbb{N}}
\newcommand{\R}{\mathbb{R}}

\newcommand{\Z}{\mathbb{Z}}
\renewcommand{\b}{\mathbf}
\renewcommand{\c}{\mathcal}
\renewcommand{\epsilon}{\varepsilon}

\renewcommand{\leq}{\leqslant}
\renewcommand{\geq}{\geqslant}
\renewcommand{\#}{\sharp}

\title
[On the application of large deviation estimates to local solubility] 
{On the application of large deviation estimates to local solubility in families of varieties}

\author{Sun Woo Park} 
\address{Max Planck Institute for Mathematics\\ 
Vivatsgasse 7
\\ 53111 Bonn, Germany}
\email{s.park@mpim-bonn.mpg.de}

\author{Efthymios Sofos} 
\address{Universit\` a di Roma Tor Vergata\\ Dipartimento di Matematica\\00133, Roma, Italy}
\email{sofos@mat.uniroma2.it}

\date{}

\begin{document} 
\begin{abstract}We apply the G\"artner--Ellis 
theorem on large deviations
to prove a weak version of the Loughran--Smeets
conjecture 
for general fibrations.
\end{abstract}

\maketitle

\setcounter{tocdepth}{1}
\tableofcontents

\section{Introduction}   \label{s:intro}   
Serre \cite{MR1075658}  investigated the probability with 
which a random Diophantine equation 
has a
$\Q$-rational point in $1990$. 
This topic was later taken up by 
Poonen--Voloch \cite{MR2029869} in $2004$, who made a 
conjecture regarding the probability with which  
hypersurfaces of fixed degree and fixed number of 
variables satisfy the Hasse principle; 
this conjecture was recently proved by 
Browning--Le Boudec--Sawin \cite{MR4564262}.

In $2016$ Loughran--Smeets \cite{MR3568035} 
put this topic into a systematic geometric framework.
Their set-up
involves a proper, smooth projective variety $V$
defined over $\Q$, an integer $n\geq 1 $ and a 
dominant morphism 
$f:V\to \P_\Q^n$ with   geometrically 
integral generic fibre. As $x$ varies over $\P^n(\Q)$
the fibres $f^{-1}(x)$ are thought of as the random 
Diophantine equations. For example, random Fermat 
curves can be described by taking   
$V$ to be 
$\sum_{i=1}^3 x_i y_i^m=0 \subset 
%ES-> OLD: \P^2\times \P^2
\P^5
%ES<-
$
and the map  $f$ sending $(x,y)$ to $x$. 
Loughran--Smeets made a very general 
conjecture regarding the probability that a 
fibre is everywhere locally soluble;
global solubility is  out of reach 
at this level of generality as there are
several kinds of 
%ES->
obstructions
%ES<-
depending 
on the nature of the fibres. Their conjecture (see
\cite[Conjecture 1.6]{MR3568035} for more details)
states that when at least
one fibre of $f$ is everywhere locally solvable and 
the fibre of $f$ over every
codimension $1$ point of $\P_\Q^n$
 has an irreducible component of multiplicity $1$, 
then  
\begin{equation}\label{conject}
\frac{c_f B^{n+1}} {(\log B)^{\Delta(f)}}
\leq 
\#\{x\in \P^n(\Q):H(x)\leq B, f^{-1}(x) \textrm{ everywhere locally soluble} \}
\leq 
\frac{c'_f B^{n+1}} {(\log B)^{\Delta(f)}}
\end{equation} holds for all large enough $B$, where 
$c_f,c'_f$ are  non-zero constants.
Here $H(x)$ is the naive Weil height on $\P^n(\Q)$ and 
$\Delta(f)$ is a Galois invariant that we recall in 
Definition \ref{def:delta}.
They proved the conjecture when $\Delta(f)=0$ in 
\cite[Theorem 1.3]{MR3568035} and verified the 
conjectured upper 
bound in all cases \cite[Theorem 1.5]{MR3568035}. 
There is a currently a heavy industry on verifying the conjecture when $V$ and $f$ 
are given by explicit equations,  
see, for instance, the references in~\cite{LRS}. To the best of authors' knowledge, however, there surprisingly seems to be a scarcity of research in seeking for an overarching strategy to understand the conjecture for general choices of $V$ and $f$.

We prove a version of \eqref{conject} in which the  local 
solubility assumption is weakened
but one that holds for all $f$ and $V$. 
The left-hand side of \eqref{conject} counts $x$ 
for which   
$f^{-1}(x)(\Q_v)\neq  \emptyset$    for all places 
 $ v \in \Omega_\Q$. We consider a weaker property 
by asking   that there are only few places 
 $v \in \Omega_\Q$  for which 
 $f^{-1}(x)(\Q_v)$ is empty. Namely, for  
 any  $\epsilon$    in $(0,\Delta(f))$ 
 we ask that \begin{equation}\label{basicproperty}
f^{-1}(x)(\Q_v)\neq \emptyset 
\textrm{ for all but }  
\epsilon\log \log H(x) 
\textrm{ places } v \in \Omega_\Q.
\end{equation}\begin{theorem}\label{thm:main1}
Let $V$ be a smooth 
projective variety over $\Q$ equipped with a dominant morphism 
$f:V\to \P^n$ with geometrically integral generic fibre and $\Delta(f)\neq 0$. Fix an arbitrary 
$0 < \epsilon < \Delta(f)$. 
Then as $B\to 
\infty$ we have $$ \#\{x\in \P^n(\Q):H(x)\leq B, 
\eqref{basicproperty}\}= 
(\log B)^{ 
-\epsilon( (\log \epsilon  )-1-\log\Delta(f))
 +o(1)} 
\cdot \frac{B^{n+1}} {(\log B)^{\Delta(f)}}.$$\end{theorem}
The right-hand side 
is compatible with  the asymptotic given  in 
Conjecture \eqref{conject} owing to  
 $\lim_{\epsilon \to 0^+}
 \epsilon( (\log \epsilon  )-1-\log\Delta(f))=0$.
 We note that the constants $c_f,c'_f$ are not visible by 
Theorem \ref{thm:main1} owing to the 
term $(\log B)^{o(1)}$.
 \subsection{Connection to probability}
The generic fibre is   geometrically integral, thus, 
by the Lang--Weil estimates and Hensel’s lemma
we see that for $x$ outside a proper Zariski closed set,
the function $$\omega_f(x) :=\#\{v \in \Omega_\Q: 
f^{-1}(x) \textrm{ has no }\Q_v\textrm{-point}\}, \ \ \ x\in \P^n(\Q)$$ 
is finite. Loughran--Sofos \cite{MR4269677}    showed that 
for $100\%$ of $x\in \P^n(\Q)$ 
the function $\omega_f(x)$ concentrates around 
 $ \Delta(f) \log \log H(x)$. More precisely, 
 they proved that 
 $$ \omega_f(x)= \Delta(f) \log \log H(x) 
 + \c Z_x \sqrt{\Delta(f) \log \log H(x)},$$ where $\c Z_x$
follows a standard normal distribution. Note that Conjecture
\eqref{conject} asks for the probability of the value 
$\omega_f(x)=0$ that   lies in the far
left of the Gaussian distribution. For comparison, 
Theorem \ref{thm:main1} corresponds to the probability of the 
red-shaded 
tail between $0$ and  $\epsilon 
\log \log H(x)$
in the following figure:  
\begin{figure}[h]
    \centering
    \includegraphics[width=\linewidth]{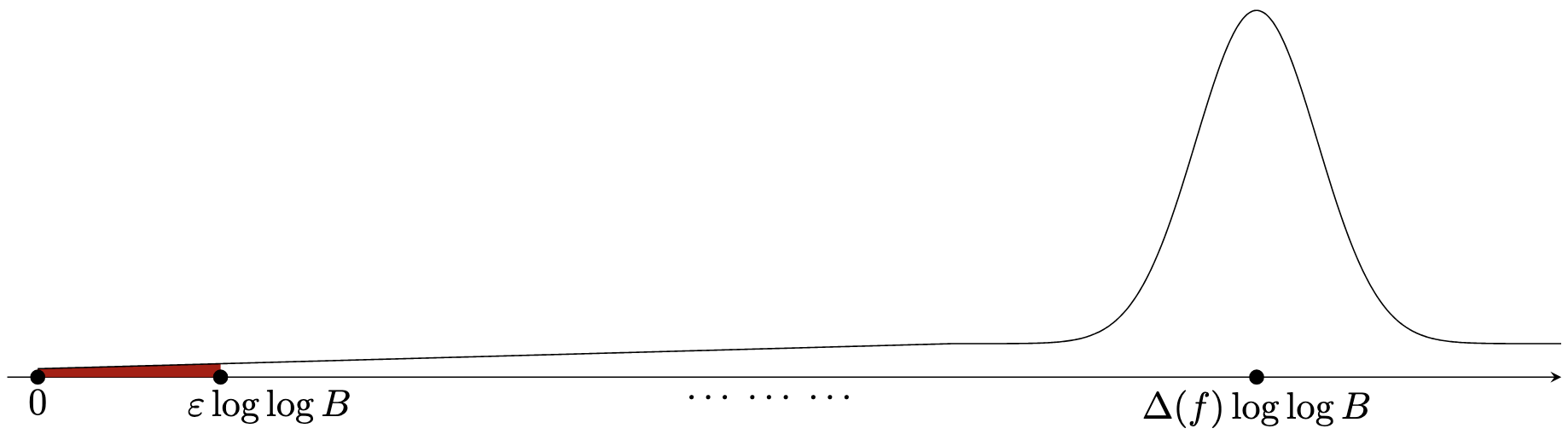}
    \caption*{Probability density function of $\omega_f(x)$.}
\end{figure}
In probability theory there are 
asymptotics for the frequency of rare events, a 
central result being the G\"artner-Ellis large deviation 
theorem \cite[Theorem 2.3.6]{MR2571413}.
The main idea behind Theorem \ref{thm:main1}
is to apply the G\"artner-Ellis result to a suitable 
probabilistic model of   $\omega_f$. \begin{remark}In order to apply the G\"artner--Ellis work (that we recall in
Lemma \ref{lem:gartellis}),
one needs asymptotics for the exponential moments 
$$\sum_{\substack{x\in \P^n(\Q), H(x)\leq B \\
f^{-1}(x) \textrm{ smooth}}} C^{\omega_f(x)}
$$ for any fixed $C\in \mathbb C$. 
This is not straightforward because  the only 
available equidistribution 
results work only when 
$p$-adic solubility conditions are imposed for primes $p$
of moderate size, see Lemma \ref{lem:distribution}. 
Indeed, the presence of $\min\{p\mid Q\}$ in the error 
term of Lemma \ref{lem:distribution} implies that adding 
the error terms for all square-free $Q$ going up to $B^\delta$
for any fixed $\delta>0$ gives an error term 
whose order of magnitude is at least as large as the main term.
In sieve theory language this means that the level of distribution is $0$.
We thus have to introduce a variant $\omega_f^\flat$ of $\omega_f$ 
that only takes into account primes of moderate size. Furthermore, 
it will be crucial to use the recent 
generalised Nair--Tenenbaum results \cite{3tors} 
regarding averages of non-negative arithmetic 
functions over arbitrary 
integer sequences.\end{remark}

\begin{remark}
There are previous results on large deviation principles for arithmetic functions,
for example, 
in the work of 
Radziwi\l\l \cite{Radzi},  
Mehrdad--Zhu \cite{MR3471275} and
Keliher--Park \cite{KP}.
In \cite{MR3471275} one can directly bound the
moments of $\omega$ by computing the analogous moments 
for the probabilistic model due to the inequality 
$$\mathbb E[Z_p]=\frac{1}{B} \left[\frac{B}{p}\right] 
\leq \frac{1}{p} \leq \mathbb E[Y_p].$$ This, or even a weaker version of this in which 
$\mathbb E[Z_p]$ is approximated by 
$\mathbb E[Y_p]$ up to an admissible error term, 
is impossible for $\omega_f$ because 
the level of distribution is $0$. 
Our proof can be easily modified to establish 
analogous large deviation results for sequences with a positive 
level of distribution. For instance, our method applies to estimating the tail bounds for the distribution of number of 
prime divisors of $F(\mathbf{x})$, where $F$ is an arbitrary 
integer polynomial in any number of variables. In this
case 
the constant $\Delta(f)$ in the asymptotic of 
Theorem \ref{thm:main1}
would be replaced by the 
number of irreducible components of $F$.
\end{remark}

\begin{remark}The Gaussian distribution results in the work of 
Loughran--Sofos \cite{MR4269677} do not directly imply 
Theorem \ref{thm:main1}. This is because knowing the 
distribution of the normalised random variable 
$(X_n-\mathbb E[X_n])/\sqrt{\textrm{Var}[X_n]}$ 
does not determine the distribution of the 
mean-scaled random variable $X_n/\mathbb E[X_n]$.
\end{remark}

 \begin{acknowledgements}Part of this work was 
 conducted while ES visited the Max Planck 
 Institute for Mathematics in Bonn in 2025, 
 where SWP holds a postdoctoral position. We 
 gratefully acknowledge the hospitality and 
 financial support provided by the Institute.
 \end{acknowledgements}
 \section{Prerequisites}
 We state all the geometric, analytic, and probabilistic prerequisites to obtain Theorem \ref{thm:main1}.
\subsection{Geometric prerequisites} 
\begin{definition}\label{def:delta} Let $f:V\to X$ be a dominant 
proper morphism of smooth irreducible varieties over a field $k$
of characteristic $0$. For each point $x\in X$ with residue 
field $\kappa(x)$, the absolute Galois group 
$ \mathrm{Gal}(\overline{\kappa(x)}/\kappa(x))$ of the residue 
field acts on the irreducible components of 
$$f^{-1}(x)_{\overline{\kappa(x)}} := f^{-1}(x) 
\times_{\kappa(x)} \overline{\kappa(x)} $$ with multiplicity $1$.
Choose some finite group $\Gamma_x$ through which this action 
factors. We define $$ \delta_x(f):=
\frac{\#\{\gamma\in \Gamma_x: \gamma \textrm{ fixes irreducible
components of } f^{-1}(x) _{\overline{\kappa(x)} }
\textrm{ with multiplicity } 1\}}{\#\Gamma_x}$$ and $$\Delta(f):= 
\sum_{D \in X^{(1)}} (1-\delta_D(f) ),$$
where $X^{(1)}$ denotes the set of codimension $1$ points of $X$.
\end{definition} For a prime $p$ let $$\sigma_p:= 
\frac{\#\{x\in \P^n(\F_p): f^{-1}(x) \textrm{ is non-split}\}}
{\#\P^n(\F_p)}.$$The term ``non-split", introduced by Skorobogatov \cite[Definition 0.1]{skoro}, refers to a scheme over a perfect field that does not contain a geometrically integral open subscheme. 
%\textbf{[SWP: Would you think it is a good idea to state explicitly what it means for $f^{-1}(x)$ to be non-split?]}
%ES thanks i added it 
Let $$c_n=\frac{2^n}{\zeta(n+1)},$$ where $\zeta$ is the Riemann zeta function. The next result is 
\cite[Proposition 3.6]{MR4594271}.
\begin{lemma}\label{lem:distribution} Keep the setting of Theorem
\ref{thm:main1}. There exist constants 
$\delta=\delta(V,f)>1$ and $A=A(V,f)>0$ with the following property.
For any square-free 
$Q\in \mathbb  N$ that is coprime to all primes $p\leq A$ 
and for all $B\geq Q^6$ we have 
$$ \# \left\{ x \in \mathbb{P}^n(\mathbb{Q}) :
\begin{array}{l}
H(x) \leq B, f^{-1}(x)\ \mathrm{smooth} \\
f^{-1}(x)(\mathbb{Q}_p) = \emptyset \ \forall\ p \mid Q 
\end{array} \right\}=c_n B^{n+1}\prod_{p\mid Q} 
\sigma_p+O\Big(\frac{\delta^{\omega(Q)} B^{n+1} }
{Q \min\{p\mid Q\}}\Big) ,$$ where the implied constant is 
independent of $B$ and $Q$. \end{lemma}The next result is
\cite[Lemma 3.2]{MR4269677}.\begin{lemma}\label{lem:polyn}
There exists a homogeneous square-free form $F\in 
\Z[x_0,\ldots, x_n]$  and $A=A(V,f)>0$ such that for all $p>A$
and $x\in \P^n(\Q)$ with $f^{-1}(x)(\Q_p)=\emptyset$ 
we have $p\mid F(x)$. \end{lemma} Let us now recall 
\cite[Lemma 3.3]{MR4269677}.\begin{lemma}
\label{lem:zelenkaoboes}Let $A,\delta$ be as in Lemma
\ref{lem:distribution}, let $F$ be as in Lemma 
\ref{lem:polyn} and assume $p>A$.
Then $$\sigma_p \leq \frac{\deg(F)}{p}.$$\end{lemma}Finally, 
we shall require \cite[Proposition 3.6]{MR4269677}. \begin{lemma}
\label{lem:mertens}There exists a constant $\beta_f$ such that 
for all $B\geq 1$ we have $$\sum_{p\leq B} \sigma_p=
\Delta(f) (\log \log B)+ \beta_f+O( (\log B)^{-1})).$$\end{lemma}
\subsection{Analytic prerequisites}
The following gives an upper bound of the right
order of magnitude for the average of a multiplicative 
function over the values of an integer polynomial
in an arbitrary number of variables.
\begin{lemma}\label{lem:nairtenenb}Assume that 
$f:\N\to [0,\infty)$ is a 
multiplicative function such that for all $m\in \N$ we have 
$f(m)\leq A_1 \tau(m)^{B_1}$, where $A_1,B_1$ are fixed constants.
Assume that $G\in \Z[x_1,\ldots, x_n]$ is an irreducible polynomial.
Then $$\sum_{\substack{ \b x \in \Z^{n}\cap[-B,B] \\ G(\b x) 
\neq 0}}f(|G(\b x )| ) \ll \frac{B^{n}}{\log B} \prod_{p\leq B^{n}}
\left(1+f(p)\frac{\#\{\b x \in \F_p^{n}: G(\b x)=0\}}{p^{n}} \right),$$ where the implied constant depends at most on 
$A_1, B_1, n$ and $G$.\end{lemma}\begin{proof} By 
\cite[Theorem 1.15]{3tors} we obtain the bound 
$$\ll  \frac{B^{n}}{\log B} \sum_{1\leq a 
\leq B^{n}} f(a) \frac{\#\{\b x \in (\Z/a\Z)^{n}:G(\b x )=0 \}}
{a^{n}},$$ where the implied constant depends at most on 
$A_1, B_1, n$ and $G$. Note that, as stated, 
\cite[Theorem 1.15]{3tors} gives a dependence of the upper bound on $f$. However, an inspection of the proof shows that 
this dependence on $f$ can be replaced by 
dependence on $A_1$ and $B_1$ instead. By multiplicativity the sum 
over $a$ is at most $ \prod_{p\leq B^{n}}(1+t_p)$, where $$t_p:=\sum_{e\geq 1} 
f(p^e) \frac{\#\{\b x \in (\Z/p^e\Z)^{n}:G(\b x )=0 \}}
{p^{en}}.$$ Let $E:=4\deg(G)$. 
By \cite[Lemma 2.8]{3tors} we see that the contribution of $e>E$ 
is$$\ll \sum_{e>E}  \frac{\tau(p^e)^{B_1} }{p^{ e /\deg(G)}}
\ll  \sum_{e>E}  \frac{e^{B_1} }{p^{ e /\deg(G)}}
\ll  \frac{E^{B_1} }{2^{ E /(2\deg(G))} }
\frac{1}{p^{ E /(2\deg(G))} }
\ll \frac{1}{p^2}.$$ By \cite[Lemma 2.8]{3tors} 
the  contribution of $2\leq e\leq E$ 
is $$\ll  \frac{\#\{\b x \in (\Z/p^2\Z)^{n}:G(\b x )=0 \}}{p^{2n}}
\sum_{2\leq e\leq E} f(p^e)
\ll \frac{E^{1+B_1}}{p^2} \ll \frac{1}{p^2} 
.$$ Hence, $$\prod_{p\leq B^{n}}(1+t_p)
\ll  \prod_{p\leq B^{n}}\left(1+f(p)\frac{\#\{\b x \in \F_p^{n}:
G(\b x)=0\}}{p^{n}}+O\left(\frac{1}{p^2}\right) \right), $$ which is 
\[\ll  \prod_{p\leq B^{n}}\left(1+f(p)\frac{\#\{\b x \in \F_p^{n}:
G(\b x)=0\}}{p^{n}} \right)
.\qedhere\]\end{proof}
\begin{lemma}\label{lem:corol} 
Let $S$ be any subset of $[2,\infty)$ 
and fix any $C>0$. Assume that $G\in \Z[x_1,\ldots, x_n]$ is an
irreducible polynomial. Then 
$$\sum_{\substack{ \b x \in \Z^{n}
\cap[-B,B] \\ G(\b x) \neq 0}} C^{\#\{p\in S:p\mid G(\b x )\}}
\ll  \frac{B^n}{\log B} \prod_{\substack{ p \in S \\ p\leq B^n }}
\left(1+\frac{\#\{\b x \in \F_p^{n}: G(\b x)=0\}}{p^{n}} \right)^C ,$$ 
where the implied constant depends at most on 
$C, n$ and $G$.\end{lemma}\begin{proof} Let 
$D\geq 0 $ be any fixed constant with  $C\leq 2^D$ and define 
$f(m)=C^{\#\{p\in S: p\mid m \}}$.  We get 
$$f(m) \leq C^{\#\{ p\mid m \}} \leq 2^{D\#\{ p\mid m \}}
\leq \tau(m)^D.$$ We can thus employ Lemma 
\ref{lem:nairtenenb} to obtain the upper  bound 
\[\ll \frac{B^n}{\log B} \prod_{\substack{ p \in S \\ p\leq B^n }}
\left(1+C\frac{\#\{\b x \in \F_p^{n}: G(\b x)=0\}}{p^{n}} \right)\leq \frac{B^n}{\log B} \prod_{\substack{ p \in S \\ p\leq B^n }}
\left(1+\frac{\#\{\b x \in \F_p^{n}: G(\b x)=0\}}{p^{n}} \right)^C
.\qedhere\]\end{proof}
The next result is taken from the work of 
Norton \cite[Lemma 4.7]{norton}.
\begin{lemma}[Norton]\label{lem:tail}For $x>0$ and $\beta>1$ we have 
$$\sum_{r\in \N \cap [\beta x,\infty) } \frac{x^r}{r!} \leq 
\frac{1}{\beta -1 } \left(\frac{\beta}{2\pi x}\right)^{1/2} 
\mathrm e^{\beta x (1-\log \beta)}.$$\end{lemma}

\begin{lemma}\label{lem:largerrr} Let $G$ be an 
integer polynomial in $n$ variables as in 
Lemma \ref{lem:nairtenenb} and 
let $N, C_1,C_2>1$ be fixed constants. Then there exists 
$y=y(G,C_i,N)>1$ such that  $$
\sum_{\substack{ \b x \in \Z^{n}\cap[-B,B] \\ G(\b x) 
\neq 0}} \sum_{r \geq y \log \log B} 
\frac{(C_1+C_2\omega(|G(\b x )|) )^r}{r!} \ll 
\frac{B^n}{(\log B)^N},$$where the implied constant 
depends at most on  $G,C_i,N$ and $n$.
\end{lemma}\begin{proof} Let $\beta >\mathrm e$. 
By Lemma \ref{lem:tail} we obtain $$
\sum_{r\geq \beta(C_1+C_2\omega(|G(\b x )|) ) } 
\frac{(C_1+C_2\omega(|G(\b x )|) )^r}{r!} \leq 
\mathrm e^{\beta (C_1+C_2\omega(|G(\b x )|) )  (1-\log \beta)}
.$$ Fix $y>1$. For those $\b x $  with 
$y \log \log B \geq \beta (C_1+C_2\omega(|G(\b x )|)  )$
we   obviously have $$
\sum_{r\geq y \log \log B } 
\frac{(C_1+C_2\omega(|G(\b x )|) )^r}{r!} \leq  
\mathrm e^{\beta (C_1+C_2\omega(|G(\b x )|) )  (1-\log \beta)}
.$$ Hence, the contribution of such $\b x $ towards the 
sum in the lemma is $$\leq  
\sum_{\substack{ \b x \in \Z^{n}\cap[-B,B] \\ 
G(\b x)  \neq 0}}
\mathrm e^{\beta (C_1+C_2\omega(|G(\b x )|) )
(1-\log \beta)}.$$
Let $f(m)=\mathrm e^{\beta C_2  \omega(m)  (1-\log \beta)}$
and note that  $0\leq f(m)\leq 1 $ because $\beta C_2   
(1-\log \beta)<0$. Hence, 
Lemma \ref{lem:nairtenenb} shows that the last sum  is 
$$\ll \frac{B^{n}}{\log B} \prod_{p\leq B^{n}}
\left(1+\mathrm e^{\beta C_2  (1-\log \beta)}
\frac{\#\{\b x \in \F_p^{n}: G(\b x)=0\}}{p^{n}} \right)
.$$ By the Lang--Weil estimates we have
$\#\{\b x \in \F_p^{n}: G(\b x)=0\}
\leq  \gamma_0 p^{n-1}$ for some 
$\gamma_0=\gamma_0(G)>0$. Thus, the product is  $$\leq 
\prod_{p\leq B^{n}}
\left(1+ \frac{1}{p} \right)
^{\gamma_0 \mathrm e^{\beta C_2  (1-\log \beta)}}
\leq  
(\gamma_G \cdot \log B))^
{\gamma_G\mathrm e^{\beta C_2  (1-\log \beta)}} $$
for some positive constant $\gamma_G$. Taking
$\beta=\beta(C_2,G,N)$ large enough so that 
$\gamma_G\mathrm e^{\beta C_2  (-1+\log \beta)} > N$
makes the bound acceptable.

For the remaining $\b x $ we have 
$y \log \log B < \beta ( C_1+C_2  \omega(|G(\b x)|))$, hence,
$$ 1<\frac{\mathrm e^{ \beta ( C_1+C_2  \omega(|G(\b x)|))}}
{(\log B)^{y}}.$$ Therefore, the corresponding  
contribution to the sum in the lemma is $$
\leq 
\sum_{\substack{ \b x \in \Z^{n}\cap[-B,B] \\ G(\b x) \neq 0}} 
\frac{\mathrm e^{  \beta (C_1+C_2  \omega(|G(\b x)|)) }}
{(\log B)^{y }}
\mathrm e^{ C_1+C_2  \omega(|G(\b x)|)}
\ll  (\log B)^{-y }
\sum_{\substack{ \b x \in \Z^{n}\cap[-B,B] \\ G(\b x) \neq 0}} 
\mathrm e^{  (\beta+1)  C_2  \omega(|G(\b x)|)  }
.$$ Applying Lemma \ref{lem:nairtenenb} as above
we find that this is  $$\ll B^{n}(\gamma_G \cdot 
\log B)^{\gamma_G(\mathrm e^{(  \beta+1)C_2}-y) }.$$
Taking large enough  $y$ so that   
$\gamma_G(\mathrm e^{(  \beta+1)C_2}-y)\leq -N$ 
completes the proof.\end{proof}
\begin{lemma}\label{lem:rromegrr}There exists a positive 
constant $\kappa$ such that 
for any $r\in \mathbb N$ and $T\geq 1 $  we have $$ 
\sum_{\substack{p_1,\ldots  , p_r \mathrm{ primes} \\
p_i \leq T \forall i}}\frac{1}{\mathrm{rad}(p_1\cdots p_r)}\ll 
3^r (\kappa+(\log r)+(\log \log T))^{r},$$ with an absolute 
implied constant.\end{lemma}\begin{proof}Let $m$ be the 
radical of $p_1\ldots p_r$ so that the sum equals 
$$\sum_{\substack{ m\in \mathbb N, \omega(m)\leq r 
\\ p\mid m \Rightarrow p\leq T} } \frac{\mu(m)^2}{m} 
\#\{p_1,\ldots, p_r \ \mathrm{ primes}: 
m=\mathrm{rad}(p_1\cdots p_r)\}.$$ 
We note that each $p_i$ can be one of $\omega(m)$ primes, 
thus, $$ \#\{p_1,\ldots, p_r \ \mathrm{ primes}: 
m=\mathrm{rad}(p_1\cdots p_r)=m\} \leq r^{\omega(m) }.$$ 
Thus, the overall bound becomes 
$$ \leq    \sum_{\substack{ m\in \mathbb N, 
\omega(m)\leq r \\ p\mid m \Rightarrow p\leq T} } 
r^{\omega(m) } \frac{\mu(m)^2}{m}.$$ For all $t\in \mathbb N$ 
and $x>1$, Hardy--Ramanujan \cite{ramanuj} proved   
$$\pi_t(x)=\#\{m\leq x: \omega(m)=t\}\leq 
\kappa_1 \frac{x}{\log x}
\frac{(\kappa+\log \log x)^{t-1} }{(t-1)!},$$ where $\kappa,\kappa_1$ are positive  absolute constants.
By partial summation we get \begin{align*}
\sum_{\substack{ 
m\leq x \\ \omega(m)=t }} \frac{1}{m}=
\frac{\pi_t(x)}{x}
+\int_2^x \frac{\pi_t(y)}{y^2}\mathrm dy
&\ll
\frac{1}{\log x}
\frac{(\kappa+\log \log x)^{t-1} }{(t-1)!}
+\frac{1}{(t-1)!}
\int_2^x 
\frac{(\kappa+\log \log x)^{t-1} }{y\log y}\mathrm dy\\
&\ll
\frac{1}{\log x}
\frac{(\kappa+\log \log x)^{t-1} }{(t-1)!}
+\frac{(\kappa+\log \log x)^{t-1}}{(t-1)!}
(\log \log x)\\&\ll \frac{(\kappa+\log \log x)^{t}}{(t-1)!}
 ,\end{align*} with absolute implied constants. Hence, 
 $$   \sum_{\substack{ m\in \mathbb N, 
\omega(m)\leq r \\ p\mid m \Rightarrow p\leq T} } 
r^{\omega(m) } \frac{\mu(m)^2}{m}\leq \sum_{t=0}^r r^t
 \sum_{\substack{ m\leq T^r\\ \omega(m)=t} } 
  \frac{1}{m}\ll 
  \sum_{t=0}^r r^t
  \frac{(\kappa+(\log r)+(\log \log T))^{t}}{(t-1)!}
.$$ 
This is at most
\[(\kappa+(\log r)+(\log \log T))^{r}
\sum_{t=1}^\infty 
  \frac{ r^t}{(t-1)!} \ll 
  3^r (\kappa+(\log r)+(\log \log T))^{r}
.\qedhere\]  \end{proof}
\subsection{Probabilistic prerequisites}
For a measurable set $A\subset \R$ let $A^\circ$ and $\overline A$ 
denote its  interior and closure respectively. 
Let us recall \cite[Theorem 2.3.6]{MR2571413}.
\begin{lemma}[G\"artner-Ellis]\label{lem:gartellis}
Let $\{X_k\}_{k\in \mathbb N}$ be a sequence of random variables on 
$\R$ and $a_k$ a strictly positive sequence with $\lim_{k\to\infty}a_k=\infty$. Assume that for any $t\in \R$ the limit 
$$ \Lambda(t):=\lim_{k\to\infty}\log 
\mathbb E[\mathrm e^{t a_k X_k}]$$ exists and that $ \Lambda(t)$ is 
differentiable at every $t$ in $\R$. Then $$I(x):=
\sup_{t\in \R}\{t x-\Lambda(t)\}$$ is non-negative and lower 
semicontinuous such that for any measurable set $A$ we have 
$$-\inf_{x\in A^\circ} I(x)\leq 
\liminf_{k\to\infty} \frac{1}{a_k} \log \mathbb P[X_k \in A]
\leq \limsup_{k\to\infty} \frac{1}{a_k} \log \mathbb P[X_k \in A]
\leq-\inf_{x\in \overline{A}} I(x).$$ 
\end{lemma}

\section{Approximating by the model} 

For $B\geq 1$ we define  
 $\Omega_B:=\{x\in \P^n(\Q): H(x) \leq B, f^{-1}(x) 
 \textrm{ smooth}\} $,
equipped with the uniform probability measure. We denote the corresponding 
probability and mean value  by $\P_B$ and $\mathbb E_B$ respectively.
Define for $x\in \Omega_B$ and a prime $p$
the function 
$$ \theta_p(x):=  \left\{
\begin{array}{ll}
1, & f^{-1}(x)(\Q_p)=\emptyset, \\
0, & \, \textrm{otherwise}, \\
\end{array}
\right. $$
where $\omega_f(x)= \sum_{p=2}^\infty \theta_p(x)$. 
We shall approximate $\omega_f(x)$ by  $$\omega_f(x)^\flat=
\sum_{t_0<p\leq t_1} \theta_p(x),$$ where 
$$ t_0=t_0(B):=(\log B)^{M}, 
\ \ \ t_1=t_1(B):=B^{\frac{1}{(\log \log B)^M}},$$
and  $M =M(B)$ is a  positive function 
that will be determined later.
\begin{lemma}[Truncating $\omega_f$]\label{lem:8} 
For any fixed $\delta>0$ and $N>1$
we have 
$$ \P_B\left[|\omega_f-\omega_f^\flat|\geq 
\delta \log \log B\right]\ll_{\delta,N} 
\frac{1}{(\log B)^N}.$$
\end{lemma}\begin{proof}
Note that $|\omega_f-\omega_f^\flat|=\omega_f-\omega_f^\flat$.
For   $y>0$  the inequality $\delta \log \log B\leq 
\omega_f-\omega_f^\flat$ implies   
$$ (\log B)^{y \delta}=\exp(y \delta \log \log B) \leq 
\exp(y (\omega_f-\omega_f^\flat)).$$ By Lemma \ref{lem:polyn}
every prime $p>A$ counted by $\omega_f-\omega_f^\flat$
must satisfy $p\mid F(x)$. Therefore, 
$$ \exp(y (\omega_f-\omega_f^\flat)) 
\leq \mathrm e^{Ay}  \exp(y (\#\{p\in S:p\mid F(x)\}))
,$$ where $S=(A,t_0] \cup [t_1,\infty)$. By the 
exponential Chebyshev inequality we infer that 
\begin{equation}\label{eq:markov} 
\P_B\left[\delta \log \log B\leq 
|\omega_f-\omega_f^\flat|\right]
\leq    \mathrm e^{Ay} 
\frac{\mathbb E_B
\left[ \mathrm e^{y \#\{p\in S: p\mid F(x)\} }\right]}
{(\log B)^{y \delta}}
.\end{equation} The mean value on the right-hand side is 
$$  \ll \frac{1}{B^{n+1} }\sum_{\substack{ \b x \in 
(\Z\cap [-B,B])^{n+1} \\ F(\b x )\neq 0 }}\mathrm e^
{y\#\{p\in S: p\mid F(\b x )\}}\ll \frac{1}{\log B} 
\prod_{\substack{ p\in S \\ p\leq B^{n+1}}}
\left(1+\frac{\#\{\b x \in \F_p^{n+1}: F(\b x )=0\}}{p^{n+1}}\right)$$by 
Lemma \ref{lem:nairtenenb}. The product term in the right-hand side 
can be written as $\Pi_1 \Pi_2$, where 
$\Pi_1$ extends over $(A,t_0]$ and $\Pi_2$ over $(t_1,B^{n+1}]$.
The Lang--Weil bound provides a positive constant $C_F$ that depends only on $F$ such that 
$\#\{\b x \in \F_p^{n+1}: 
F(\b x )=0\}\leq C_F p^n $. Hence, 
$$ \Pi_1 \ll \prod_{A<p\leq t_0} (1+p^{-1})^{C_F}\ll
(\log \log t_0)^{C_F} \ll (\log \log \log B)^{C_F}
$$ and $$ \Pi_2 \ll \prod_{t_1<p\leq B^{n+1} } (1+p^{-1})^{C_F}\ll
\left(\frac{\log B}{\log t_1}\right)^{C_F}=
 (\log \log B )^{MC_F}.$$ Hence, by \eqref{eq:markov} we get 
$$  \P_B\left[\delta \log \log B\leq 
\omega_f-\omega_f^\flat\right]
\ll \frac{ (\log \log \log B)^{C_F} (\log \log B)^{MC_F}}
{(\log B)^{y \delta}}\ll (\log B)^{-y \delta/2}.$$
Taking $y=2N/\delta$ concludes the proof. 
\end{proof} Recall that by Lemma \ref{lem:mertens} we have $\sigma_p\leq \deg(F)/p$, where $A$ is the constant 
introduced  in Lemma \ref{lem:zelenkaoboes} and 
$F$ is as in Lemma \ref{lem:polyn}.
Hence, for a prime $p>\max\{A,\deg(F)\}$ we have  
$\sigma_p<1$. Now for $p>\max\{A,\deg(F)\}$ 
we let $Y_p$ be a random variable that takes the value 
$1$ with probability $\sigma_p$ and that takes the value 
$0$ with probability $1-\sigma_p$. We assume that the random variables $\{Y_p\}_{p > \max\{A, \deg(F)\}}$
are independent; the existence of random variables with 
these properties on a some probability space is 
guaranteed by Kolmogorov’s extension theorem. The sum 
$$ S_B:=\sum_{t_0<p\leq t_1} Y_p $$
is a model for $\omega_f^\flat$. 
 \begin{lemma}\label{lem:prelimin0007}
For all $B\geq 1 $ and $r\in \mathbb N$ we have$$ \mathbb E[S_B^r]\leq2  (K (\deg(F)\log \log B)^r,$$ where $K$ is a positive absolute constant and $F$ is as in 
Lemma \ref{lem:polyn}.\end{lemma}\begin{proof}We have $$ \mathbb E[S_B^r]=
\sum_{p_1,\ldots, p_r\in (t_0,t_1]} 
\mathbb E[Y_{p_1}\cdots Y_{p_r}]=
\sum_{p_1,\ldots, p_r\in (t_0,t_1]} 
\prod_{p\mid \textrm{rad}(p_1\cdots p_r)}
\sigma_p\leq (\deg(F))^r
\sum_{p_1,\ldots, p_r\leq B} 
\frac{1}{\textrm{rad}(p_1\cdots p_r)}
$$ by Lemma \ref{lem:zelenkaoboes}. 
One could now use Lemma 
\ref{lem:rromegrr}. However, such a strategy is not 
efficient for large $r$, as it results in having to bound the 
tail of $\sum_r (\log r)^r/r!$. An alternative bound is 
$$\sum_{j=1}^r \left(\sum_{p\leq B} \frac{1}{p}\right)^j
\leq \sum_{j=1}^r (K\log \log B)^j \leq
 (K\log \log B)^r \sum_{t=0}^{r-1}  (K\log \log B)^{-t}
$$ for some absolute positive constant $K$ by Mertens' theorem.
We can assume that $K\log \log B>2$ with no loss of 
generality, thus, 
the sum over $t$ is at most $\sum_{t\geq 0}2^{-t}\leq 2 $.\end{proof}
\begin{lemma}\label{lem:9a01ststep} There exist  
constants $\kappa,K>0,$ and $\delta=\delta(f,V)>0$ such that for all
$B\geq 1$ and $r \in \mathbb{N}$ in the range 
$r\leq (\log B)/(6 \log t_1)$ we have 
$$\mathbb E_B[\mathrm 
( \omega_f^\flat)^r]- \mathbb E[S_B^r]
\ll \frac{(3\delta)^r}{  (\log B)^M } 
  (\kappa+(\log r)+(\log \log B))^{r} 
  +\frac{  (K (\deg(F)\log \log B)^r}{B^{1/2}}
,$$ where the implied 
constants are independent of $r$ and $B$.\end{lemma}
\begin{proof} We have $$( \omega_f^\flat)^r=
\sum_{p_1,\ldots,p_r\in (t_0,t_1]} \theta_{p_1} 
\cdots \theta_{p_r}.$$ Hence, $\mathbb E_B[( \omega_f^\flat)^r]$ equals 
$$\sum_{p_1,\ldots,p_r\in (t_0,t_1]} 
\frac{\#\{x\in \P^n(\Q): H(x) \leq B, f^{-1}(x) 
 \textrm{ smooth}, f^{-1}(x)(\mathbb{Q}_{p}) = 
 \emptyset \ \forall p\mid \textrm{rad}(p_1\cdots p_r) \}}
 {\#\{x\in \P^n(\Q): H(x) \leq B, f^{-1}(x)
 \textrm{ smooth}\} },$$ where the radical of 
 $p_1\cdots p_r$ is used because there may be repetitions  
 among the primes $p_i$. We now employ 
 Lemma \ref{lem:distribution}. To do this, we first 
 verify its assumption $Q\leq B^{1/6}$ with 
 $Q=\textrm{rad}(p_1\ldots p_r)$. We have 
 $Q\leq t_1^r$, hence, our assumption $t_1^{6r} \leq B$ ensures
 that we can use  Lemma \ref{lem:distribution}.
Together with 
$$ \#\{x\in \P^n(\Q): H(x) \leq B, f^{-1}(x)\textrm{ smooth}\} =c_n B^{n+1}+O(B^{n+1/2}),$$ 
 we can rewrite 
 $\mathbb{E}_B[(\omega_f^\flat)^r]$ as$$
 (1+O(B^{-1/2}))
 \left( \sum_{p_1,\ldots,p_r\in (t_0,t_1]} 
\prod_{p\mid \textrm{rad}(p_1\cdots p_r) } 
\sigma_p\right)
+O\left(
\frac{\delta^r}{  t_0 }
\sum_{p_1,\ldots,p_r\in (t_0,t_1]} 
\frac{1 }
{\textrm{rad}(p_1\cdots p_r) }\right).$$ 
By Lemma \ref{lem:rromegrr} the second $O$-term is
$$\ll \frac{(3\delta)^r}{  t_0 }(\kappa+(\log r)+(\log \log B))^{r}.$$ By independence   we have 
\[\sum_{p_1,\ldots,p_r\in (t_0,t_1]} 
\prod_{p\mid \textrm{rad}(p_1\cdots p_r) } 
\sigma_p=\mathbb E[S_B^r].\] Hence, 
Lemma \ref{lem:prelimin0007} shows that the first 
  $O$-term is\[\ll B^{-1/2} \mathbb E[S_B^r]
  \ll \frac{  (K \deg(F)\log \log B)^r}{B^{1/2}}
  .\qedhere\]
  \end{proof}
 \begin{lemma}\label{JMFC=SDG}
Fix positive constants $\lambda$ and $N$. Then there exists
$z=z(\lambda,N,f,V)$ such that for all $B\geq 1 $
and $r\in \mathbb N$ we have 
$$
\sum_{r> z \log \log B} \frac{\lambda^r}{r!} \mathbb E[S_B^r]
\leq \frac{2}{(\log B)^N}$$
\end{lemma}\begin{proof} By Lemma \ref{lem:prelimin0007}
we get the bound $$  \leq 2 \sum_{r> z \log \log B} 
\frac{(\lambda  \deg (F) K\log \log B)^r }{r!}.$$ 
Assume that $z>\mathrm e  \lambda  \deg (F) K$ and apply 
Lemma \ref{lem:tail} with 
$x=\lambda  \deg (F) K\log \log B$ and 
$\beta=z/( \lambda  \deg (F) K)$. We obtain 
the bound $  \leq 2 
\exp(\beta x (1-\log \beta)) $. This is 
$\leq 2(\log B)^{-N}$ if 
$$N \leq z \log \frac{z}{ \mathrm e\lambda  \deg (F) K}.$$
This inequality holds when $z$ is sufficiently large.
\end{proof}

\begin{lemma}\label{lem:9a} Assume that 
$M(B) \geq 2$.
For each $B>1, 
t\in \mathbb R$ and $N>0$ there exists a 
positive constant 
$\gamma=\gamma(f,V,t,N)$ such that 
$$\mathbb E_B[\mathrm e^{t \omega_f^\flat}]-
\mathbb E[\mathrm e^{t S_B}] 
\ll (\log B)^{-\min\{N,M(B)-\gamma\}},$$ where the implied 
constant depends at most on $f,V,t$ and $N$.\end{lemma}
\begin{proof} We have $$\mathbb E_B[\mathrm 
e^{t \omega_f^\flat}]=
\frac{1}{\#\{x\in \P^n(\Q): H(x) \leq B, f^{-1}(x) 
 \textrm{ smooth}\} } 
\sum_{\substack{ x\in \P^n(\Q), H(x) \leq B\\ 
f^{-1}(x) \textrm{ smooth} }}
 \sum_{r=0}^\infty \frac{t^r}{r!}\omega_f^\flat(x)^r.$$ By Lemma \ref{lem:polyn}
we have $\omega_f^\flat(x) \leq A+\omega(|F(x)|)$.
We may therefore employ Lemma \ref{lem:largerrr}
with $G=F, C_1=A$ and $C_2=1$ to deduce that 
for each  $N>1$   there is 
$y=y(A,F,N,t)$ such that 
$$ \frac{1}{\#\{x\in \P^n(\Q): H(x) \leq B, f^{-1}(x) 
 \textrm{ smooth}\} } 
\sum_{\substack{ x\in \P^n(\Q), H(x) \leq B\\ 
f^{-1}(x) \textrm{ smooth} }}
 \sum_{r\geq y \log \log B} \frac{t^r}{r!}\omega_f^\flat(x)^r
 \ll \frac{1}{(\log B)^N},$$ where the implied constant 
depends on $N,A$ and $F$. We have $$\mathbb E[\mathrm 
e^{t S_B}]=
 \sum_{r=0}^\infty \frac{t^r}{r!}\mathbb E[S_B^r].$$ 
By Lemma \ref{JMFC=SDG} there exists $z=z(f,V,N,t)$ such that 
$$  \sum_{r> z \log \log B} \frac{t^r}{r!}\mathbb E[S_B^r]
\ll \frac{1}{(\log B)^N}.$$  Taking $\varpi=\max\{y,z\}$, 
we deduce that there is a constant $\varpi$ such that 
\begin{equation}\label{eq:need more espresso}
\mathbb E_B[\mathrm e^{t \omega_f^\flat}]-
\mathbb E[\mathrm e^{t S_B}] 
=\sum_{r\leq  \varpi \log \log B} \frac{t^r}{r!}
(\mathbb E_B[(\omega_f^\flat)^r] - \mathbb E[S_B^r])
+O\left( \frac{1}{(\log B)^N} \right).\end{equation} The assumption 
$r\leq (\log B)/(6 \log t_1)$ of Lemma \ref{lem:9a01ststep}  
is met due to 
$t_1(B)=B^{\frac{1}{(\log \log B)^M}}$
and our assumption $M\geq 2 $. 
Thus, the sum over $r$ in \eqref{eq:need more espresso} is 
$$ \ll \sum_{r\leq  \varpi \log \log B} \frac{|t|^r}{r!}
\left( \frac{(3\delta)^r}{  (\log B)^{M(B)} } 
  (\kappa+(\log r)+(\log \log B))^{r} 
  +\frac{  (K (\deg(F)) \log \log B)^r}{B^{1/2}}\right).$$
  Using $r\leq  \varpi \log \log B$ we find  
a positive constant $\gamma=\gamma(\varpi,t,K,F,\delta,\kappa)$
such that the above is \[ \ll 
\frac{1}{  (\log B)^{M(B)} } 
\sum_{r\leq  \varpi \log \log B} 
\frac{(\gamma\log \log B)^r}{r!}\leq  
\frac{(\log B)^{\gamma}}{  (\log B)^{M(B)} } 
.\qedhere\] \end{proof}\begin{lemma}
\label{lem:9and0} Assume that $M(B)\to\infty$.
For  any fixed   $t\in \mathbb R$ 
we have 
$$\lim_{B\to\infty}
\frac{|(\log \mathbb E_B[\mathrm e^{t \omega_f^\flat}])-(
\log \mathbb E[\mathrm e^{t S_B}])|}
{\log \log B}=0.$$\end{lemma}
\begin{proof}Lemma \ref{lem:9a} yields 
$$\mathbb E_B[\mathrm e^{t \omega_f^\flat}]=
\mathbb E[\mathrm e^{t S_B}]  
+O\left( \frac{1}{(\log B)^N} \right)$$ 
for any fixed $N>0$. When $t\geq 0$ we have  
$\mathrm e^{t \omega_f^\flat} \geq 1 $, thus, 
$\mathbb E_B[\mathrm e^{t \omega_f^\flat}] \geq 1$. 
This means that $$\frac{
\mathbb E[\mathrm e^{t S_B}]  }{
\mathbb E_B[\mathrm e^{t \omega_f^\flat}]}=1
+O\left( \frac{1}{
\mathbb E_B[\mathrm e^{t \omega_f^\flat}](\log B)^N} \right)=1+O\left( \frac{1}{(\log B)^N} \right).$$ Taking 
logarithms concludes the proof in this case.
When $t<0$ the same approach does not work.
However, if we manage to show that there exists $c>0$ that is independent of $B$ such that 
\begin{equation}\label{nasty}
\mathbb E_B[\mathrm e^{t \omega_f^\flat}] >
\frac{1}{(\log B)^c},\end{equation} then for 
each fixed $N>0$ we obtain 
 $$\frac{
\mathbb E[\mathrm e^{t S_B}]  }{
\mathbb E_B[\mathrm e^{t \omega_f^\flat}]}=1
+O\left( \frac{1}{
\mathbb E_B[\mathrm e^{t \omega_f^\flat}](\log B)^N} \right)=1+O\left( \frac{1}{(\log B)^{N-c}} \right).$$
Choosing any  $N>c$ and taking logarithms 
concludes the proof also when  $t<0$.  

Let us now verify \eqref{nasty}.
We have $1=\mathrm e^{t/2 \omega_f^\flat} \cdot
\mathrm e^{-t/2 \omega_f^\flat}$, hence 
Cauchy's inequality gives
$$1=\mathbb E_B[\mathrm e^{t/2 \omega_f^\flat}\cdot 
\mathrm e^{-t/2 \omega_f^\flat}]
\leq 
\mathbb E_B[\mathrm e^{t \omega_f^\flat}]^{1/2}
\mathbb E_B[\mathrm e^{-t \omega_f^\flat}]^{1/2}
.$$ It thus remain to show 
 $\mathbb E_B[\mathrm e^{-t \omega_f^\flat}]
 \ll (\log B)^{c'}$  for some constant $c'>0$ 
 independent of $B$. For this we define the 
 multiplicative function $g(m)=
 \mathrm e^{-t \omega(m)}$ and note that 
 $g\geq 0$ as $t $ is negative. By Lemma
 \ref{lem:polyn} we have 
 $$\mathbb E_B[\mathrm e^{-t \omega_f^\flat}]
 \ll B^{-n-1}\sum_{\substack{ \b x \in 
 (\Z\cap [-B,B])^{n+1} \\  F(\b x) \neq 0}}
 g(|F(\b x )| ) .$$ Applying Lemma \ref{lem:corol}
 yields $$\mathbb E_B[\mathrm e^{-t \omega_f^\flat}]
 \ll  
 \prod_{\substack{ p\leq B^{n+1} }}
\left(1+\frac{\#\{\b x \in \F_p^{n+1}: F(\b x)=0\}}{p^{n+1}} \right)^{\mathrm e^{-t}} .$$ By the 
Lang--Weil estimate we have 
$\#\{\b x \in \F_p^{n+1}: F(\b x)=0\}\ll p^n$,
thus the above becomes $$\mathbb E_B[\mathrm e^{-t \omega_f^\flat}]
 \ll  
 \prod_{\substack{  p\leq B^{n+1} }}
\left(1+\frac{O(1)}{p} \right)^{\mathrm e^{-t}}\ll
(\log B)^b,$$ where $b$ is a positive constant that only depends on $t$ and $F$. 
One can then take $c''=2b$.\end{proof}

 \section{Applying probability theory}
The goal of this section is to use Lemma \ref{lem:gartellis} to prove Theorem \ref{thm:main1}.
\begin{lemma}\label{lem:9b}
Assume that $M(B)\to \infty $ and 
$M(B)\leq \log \log B$. Then 
for each $t \in \mathbb{R}$ we have
$$ \lim_{B \to \infty} \frac{\log \mathbb{E} 
\left[ \mathrm e^{t S_B} \right]}{\log 
\log B} = \Delta(f) ( \mathrm e^t - 1). $$
\end{lemma}
\begin{proof}
For each prime $p$ we have 
$\mathbb{E} \left[\mathrm  e^{t Y_p} \right] 
= \mathrm e^t \sigma_p+1-\sigma_p$, hence, 
$$ \mathbb{E} \left[\mathrm  e^{t S_B} \right] 
= \prod_{t_0 < p \leq t_1} 
\mathbb{E} \left[ \mathrm  e^{t Y_p} \right] =
\prod_{t_0 < p \leq t_1} 
\left( 1+\sigma_p(\mathrm e^t  -1)\right)$$
owing to the independence of the 
random variables $\{Y_p\}_{p > \max\{A, \deg(F)\}}$.
Taking logarithms,
using the Taylor expansion of $\log(1+\epsilon)$, 
and applying the bound 
$\sigma_p\ll 1/p$ from Lemma \ref{lem:zelenkaoboes}
leads to $$ \log \mathbb{E} \left[ e^{t S_B} \right] 
= \sum_{t_0 < p \leq t_1} \log 
\left( 1+\sigma_p(\mathrm e^t  -1)\right)=(\mathrm e^t  -1)
 \sum_{t_0 < p \leq t_1}  
 \sigma_p+O\left(\sum_{p\geq 2} p^{-2}\right),$$
where the implied constant depends only on $t$.
By Lemma \ref{lem:mertens}, we obtain 
$$ \log \mathbb{E} \left[\mathrm  e^{t S_B} \right] 
=\Delta(f) \log \frac{\log t_1}{\log t_0}
+O(1).$$ Recalling the definitions of $t_0$ and 
$t_1$
we obtain 
$$\log \frac{
{ (\log \log B)^{-M(B)} 
(\log B)} }{M(B) \log \log B}=
\log \log B+O((\log \log \log B)^2)
.$$ This proves $ 
\log \mathbb{E} \left[ \mathrm e^{t S_B} \right] 
=\Delta(f) \log \log B+O((\log \log \log B)^2)$, which suffices. 
\end{proof}
\subsection*{Proof of Theorem \ref{thm:main1}}
\begin{proof}
%ES-> 
We let $M(B)=\max\{2,\log \log B\}$ 
so that the assumptions of
Lemmas \ref{lem:9a}-\ref{lem:9b} are met.
%ES<-
    For each $B \geq 3$ we let $a_B:=\Delta(f)\log \log B$
    and  we define  the random variable
  $X_B$   as$$X_B(x) := \frac{\omega_f(x)^\flat}
{\Delta(f)\log \log B}, \ \ \ x\in \Omega_B.$$
    Then by combining Lemmas \ref{lem:9and0} 
    and \ref{lem:9b}, we find that 
    $$
    \Lambda(t) = \lim_{B \to \infty} \frac{1}{a_B} 
    \log \mathbb{E}[\mathrm e^{t a_B X_B}] = 
    \lim_{B \to \infty} \frac{1}{\Delta(f) \log \log B} 
    \log \mathbb{E}_B[\mathrm e^{t \omega_f^\flat}] 
    = \mathrm e^t - 1,$$ which is differentiable at every $t \in \mathbb{R}$.
    The rate function $I(x)$ is defined as 
    $$
    I(x) := \sup_{t \in \mathbb{R}}\{tx - \Lambda(t)\} = \sup_{t \in \mathbb{R}}\{tx - \mathrm e^t + 1\}.$$  It is easy to check that $I(1) = 0$.
For $x >0$, we have $I(x) = x (\log x) + 1 - x$.
For $x < 0$, $I(x)$ is $+\infty$.  By Lemma \ref{lem:gartellis}, for any measurable set $A$ we have
    \begin{align*}
    - \inf_{x \in A^\circ} I(x) &\leq \liminf_{B \to \infty} \frac{1}{\Delta(f) \log \log B} \log \mathbb{P}_B \left[ X_B \in A \right] \\
    &\leq \limsup_{B \to \infty} \frac{1}{\Delta(f) \log \log B} \log \mathbb{P}_B \left[ X_B \in A \right] \leq -\inf_{x \in \overline{A}} I(x).
    \end{align*}
    Choose $A := [0, \epsilon)$ for some fixed $0 < \epsilon < 1$. Then  $$ - \inf_{x \in A^\circ} I(x) 
= -(\epsilon \log \epsilon + 1 - \epsilon) = 
-\inf_{x \in \overline{A}} I(x).$$ Hence, we have
\begin{equation}\label{eq:the tea room}
\lim_{B \to \infty} \frac{1}{\Delta(f) \log \log B} 
\log \mathbb{P}_B \left[ \frac{\omega_f^\flat}{\Delta(f) 
\log \log B} < \epsilon \right] = -(\epsilon \log \epsilon 
+ 1 - \epsilon).\end{equation}

For fixed $c>0$ define the sets 
\begin{align*}
\c A_c&:=\left\{x\in \Omega_B: 
\frac{\omega_f}{\Delta(f) \log \log B}<c \right\}, \\
\c A^\flat_c&:=\left\{x\in \Omega_B: 
\frac{\omega_f^\flat}{\Delta(f) \log \log B}<c \right\}, \\
\c E_c&:=\left\{x\in \Omega_B: 
\frac{|\omega_f-\omega_f^\flat|}{ \Delta(f) \log \log B}
\geq c \right\}.\end{align*}
Since $\omega_f^\flat\leq \omega_f$ we have $\c A_\epsilon
\subset \c A^\flat_\epsilon$, hence, by \eqref{eq:the tea room} 
we obtain 
\begin{equation}\label{eq:uperbnd}\limsup_{B\to\infty}
\frac{\log
\mathbb{P}_B \left[ \frac{\omega_f}{\Delta(f) 
\log \log B} < \epsilon \right] }{\Delta(f)\log \log B}
\leq  - I(\epsilon) .\end{equation}
Fix any $\delta\in (0,\epsilon)$.
If $\omega_f^\flat/(\Delta(f) \log \log B)< \epsilon-\delta$
then we either have 
$\omega_f/(\Delta(f) \log \log B)< \epsilon$
or 
$|\omega_f-\omega_f^\flat|/(\Delta(f) \log \log B)\geq \delta$,
hence, $\c A^\flat_{\epsilon-\delta} 
\subset \c A_{\epsilon}\cup 
\c E_{\delta} $. Thus, by \eqref{eq:the tea room} 
and Lemma \ref{lem:8} we get  $$
(\log B)^{-\Delta(f)I(\epsilon-\delta) +o(1)}+
O\left(\frac{1}{(\log B)^N}\right) 
\leq\mathbb{P}_B \left[ \frac{\omega_f}{\Delta(f) 
\log \log B} < \epsilon \right]
,$$ for any fixed $N>1$. This is the same as writing 
$$
(\log B)^{-\Delta(f)I(\epsilon-\delta) +o(1)}\left(1+
O\left(\frac{1}{(\log B)^K}\right) \right)
\leq\mathbb{P}_B \left[ \frac{\omega_f}{\Delta(f) 
\log \log B} < \epsilon \right]
,$$ for any fixed $K>1$. Taking logarithms leads to 
$$-I(\epsilon-\delta) \leq \liminf_{B\to\infty} 
\frac{ \log \mathbb{P}_B \left[ \frac{\omega_f}{\Delta(f) 
\log \log B} < \epsilon \right]}{\Delta(f)\log \log B}
.$$ We let $\delta\to 0_+$ and combine 
with \eqref{eq:uperbnd} to deduce that for any 
$\epsilon \in (0,1)$ one has 
\begin{equation}\label{amplification} \#\{x\in \P^n(\Q):H(x)\leq B, 
\omega_f(x)<\epsilon \Delta(f) \log \log B\}= 
(\log B)^{-\Delta(f) \epsilon( (\log \epsilon)-1) 
 +o(1)} 
\cdot \frac{B^{n+1}} {(\log B)^{\Delta(f)}}.\end{equation}
Finally, we want to prove the analogous asymptotic with $\log \log B$ replaced by $\log \log H(x)$ in the left-hand side. We plainly have 
\begin{align*} &\#\{x\in \P^n(\Q):H(x)\leq B, 
\omega_f(x)<\epsilon \Delta(f) \log \log H(x)\}
\\ \leq 
  &\#\{x\in \P^n(\Q):H(x)\leq B, 
\omega_f(x)<\epsilon \Delta(f) \log \log B\}= 
(\log B)^{-\Delta(f) \epsilon( (\log \epsilon)-1) 
 +o(1)} 
\cdot \frac{B^{n+1}} {(\log B)^{\Delta(f)}},
\end{align*}
hence, 
\begin{equation}\label{refinement}\limsup_{B\to\infty} 
\frac{ \log \mathbb{P}_B \left[ \frac{\omega_f}{\Delta(f) 
\log \log H(x)} < \epsilon \right]}{\Delta(f)\log \log B}
\leq -I(\epsilon).\end{equation}
For the lower bound we note that $$
\#\{x\in \P^n(\Q):H(x)\leq B^{1/2}\}=O(B^{\frac{n+1}{2}})
.$$When $H(x)>B^{1/2}$ we have 
$\log \log B \leq (\log \log H(x)) +(\log 2)$, hence, 
for any $\delta\in (0,\epsilon)$ we obtain \begin{align*}
&\#\{x\in \P^n(\Q):H(x)\leq B, 
\omega_f(x)<(\epsilon-\delta) \Delta(f) \log \log B\}
\\
=O(B^{\frac{n+1}{2}})+
&\#\{x\in \P^n(\Q):B^{1/2}<H(x)\leq B, 
\omega_f(x)<(\epsilon-\delta) \Delta(f) \log \log B\}
\\
\leq O(B^{\frac{n+1}{2}})+
&\#\{x\in \P^n(\Q):B^{1/2}<H(x)\leq B, 
\omega_f(x)<\Delta(f) (\epsilon \log \log B
- \log 2)\}
\\
\leq O(B^{\frac{n+1}{2}})+
&\#\{x\in \P^n(\Q):B^{1/2}<H(x)\leq B, 
\omega_f(x)<\Delta(f) \epsilon \log \log H(x)\}
\\
\leq O(B^{\frac{n+1}{2}})+
&\#\{x\in \P^n(\Q):H(x)\leq B, 
\omega_f(x)<\Delta(f) \epsilon \log \log H(x)\}
.\end{align*} Hence, 
$$\P_B\left[\frac{\omega_f}{\Delta(f) \log \log B}<\epsilon-\delta \right]
\leq O(B^{-\frac{n+1}{2}})
+\P_B\left[\frac{\omega_f}{\Delta(f) \log \log H(X)}<\epsilon \right].$$ Dividing by $\Delta(f)\log \log B$ and applying
\eqref{amplification} yields 
$$-I(\epsilon-\delta) \leq 
\liminf_{B\to\infty}
\frac{\log
\mathbb{P}_B \left[ \frac{\omega_f}{\Delta(f) 
\log \log H(x)} < \epsilon \right] }{\Delta(f)\log \log B}
.$$ Letting 
$\delta\to0_+$ and using the continuity of $I$ at $\epsilon$ proves $$-I(\epsilon) \leq 
\liminf_{B\to\infty}
\frac{\log
\mathbb{P}_B \left[ \frac{\omega_f}{\Delta(f) 
\log \log H(x)} < \epsilon \right] }{\Delta(f)\log \log B}
.$$ Together with \eqref{refinement} this shows that $$
\lim_{B\to\infty}
\frac{\log
\mathbb{P}_B \left[ \frac{\omega_f}{\Delta(f) 
\log \log H(x)} < \epsilon \right] }{\Delta(f)\log \log B}
= -I(\epsilon).$$ 
Finally, replacing $\epsilon \Delta(f)$ by
$\epsilon$ brings this  into the form 
that is claimed in Theorem \ref{thm:main1}.\end{proof}

\end{document}